\documentclass[11pt]{amsart}
\usepackage{amsmath,amsthm,amssymb, amsfonts, amscd}
\usepackage{enumerate}
\usepackage{color}

\usepackage{braket}

\newtheorem{thm}{Theorem}[section]
\newtheorem{cor}[thm]{Corollary}
\newtheorem{lem}[thm]{Lemma}
\newtheorem{prop}[thm]{Proposition}
\newtheorem{conj}[thm]{Conjecture}

\theoremstyle{plain}
\numberwithin{equation}{thm}
\theoremstyle{remark}

\def\P{\mathbb P}

\title{Backward Orbit Conjecture for the Powering Map over Global Fields}
\author{Vijay  Sookdeo}
\address{
Vijay Sookdeo\\
Department of Mathematics\\
The Catholic University of America\\
Washington, DC 20064 \\
}

\begin{document}

\begin{abstract}
We show that the backward orbit conjecture is true for powering map 
$\phi(z)=z^d$ over a function field $K$ with a finite field of constants, and 
when $d$ is relatively prime to the characteristic of $K$.
\end{abstract}

\maketitle

\section{Introduction}
Let $K$ be a finite extension of $\mathbb{F}_p(t)$, and $\phi:\P^1 \to \P^1$ be 
a rational map of degree $d\ge 2$ defined over $K$.  Write $\overline{K}$ for 
the separable algebraic closure of $K$, and  $\phi^n$ for the 
$n$th iterate of $\phi$, . For a  point $\beta \in \P^1$, let 
$\phi^+(\beta)=\{\beta, \phi(\beta), \phi^2(\beta), \dots \}$ be the 
\emph{forward orbit} of $\beta$ under $\phi$, and let $$\phi^-(\beta) = 
\bigcup_{n\ge 0} \phi^{-n}(\beta) \subset \mathbb P^1(\overline K)$$ be the 
\emph{backward orbit} of $\beta$ under $\phi$.  We say $\beta$ is 
$\phi$-\emph{preperiodic} if and only if $\phi^+(\beta)$ is finite.

\begin{conj}\label{bc}
If $\alpha \in \mathbb{P}^1(K)$ is not $\phi$-preperiodic, then for any 
$\beta\in P^1(K)$, $\phi^-(\beta)$ contains at most finitely many points in 
$\mathbb{P}^1(\overline K)$ which are $S$-integral relative $\alpha$. 
\end{conj}

This conjecture was stated over number fields $K$ in \cite{sook} and 
proven true for the powering map by bounding the number of irreducible factors 
of  $z^n-\alpha$ for $\alpha\in K$, and using Siegel's Theorem for integral 
points on $\mathbb{G}_m(K)$.  This approach cannot work for function fields $K$ 
since there is no analogous Siegel's Theorem in this case.  By using ideas 
developed in \cite{bir}, which involve showing that the limit 
$$\sum_{v\in M_K}\lim_{n\to\infty}\frac{1}{[K(\gamma_n):\mathbb 
F_p(t)]}\sum_{\sigma: 
K(\gamma_n)/K\to\overline{K}_v}\log|\sigma(\gamma_n)-\alpha|_v $$
converges to the height of $\alpha$ when $\gamma_n\in\phi^-(\beta)$, we are able 
to establish the following result.

\begin{thm}\label{main}
 Let $S$ a finite set of places of $K$; and $\phi(z)=z^d$, with $d\ge 2$ 
relatively prime to the characteristic $K$.  If $\alpha \in \mathbb P^1(K)$ is 
not $\phi$-preperiodic, then $\phi^{-}(\beta)$ contains at most finitely many 
points in $\mathbb{P}^1(\overline K)$ which are $S$-integral relative to 
$\alpha$. 
\end{thm}

An immediate corollary, based on the functorial property of integrality, is a 
result of Conjecture \ref{bc} for Chebychev polynomials (see Section 
\ref{cheby}).

\section{Heights and Integrality}
\subsection{Heights}
Let $\alpha\in \mathbb F_p(t)\setminus \{0\}$.  When $\frak p$ is a prime ideal 
of $\mathbb F_p[t]$, we define the $v_{\frak p}(\alpha)$, the 
\emph{$\frak p$-adic valuation} of $\alpha$, to be the largest integer $n$ such 
that $\alpha\in\frak{p}^n$.  Writing $\alpha = a/b$, with $a,b\in \mathbb 
F_p[t]$ relatively prime polynomials, we define the valuation 
$v_{\infty}(\alpha)$ to be $\deg(b)-\deg(a)$.  Set $v_{\frak p}(0)=\infty$ when 
$\frak p$ is a prime ideal of $\mathbb F_p[t]$ or $\frak p = \infty$.

Fix a real number $0<\varepsilon <1$.  We define the absolute value 
$|\cdot|_{\frak p}$ to be $$|\alpha|_{\frak p} = \varepsilon^{v_{\frak 
p}(\alpha)},$$ where $\frak p$ is a prime ideal of $\mathbb F_p[t]$ or $\frak p 
= \infty$.  This defines a complete set of inequivalent absolute values on 
$\mathbb F_p(t)$, which we denote by $M_{\mathbb F_p(t)}$.  Furthermore, each 
of absolute value on $\mathbb F_p(t)$ is non-archimedean in that the 
\emph{ultrametric inequality} if satisfied: $$|\alpha + \beta|_{\frak p} \le 
\max(|\alpha|_{\frak p}, |\beta|_{\frak p}).$$
It will be useful to note that ultrametric inequality implies that if 
$|\alpha|_{\frak p} \not= |\beta|_{\frak p}$, then $|\alpha + \beta|_{\frak p} 
= \max(|\alpha|_{\frak p}, |\beta|_{\frak p}).$  We have that $M_{\mathbb 
F_p(t)}$ satisfies the \emph{product formula}:  If $\alpha \in 
\mathbb{F}_p(t)\setminus \{0\},$ then $$\prod_{\frak p \in M_{\mathbb F_p(t)}} 
|\alpha|_{\frak p} =1.$$ 

Let $K$ be a finite extension of $F=\mathbb F_p(t)$.  A \emph{place} $v$ of $K$ 
is a set of equivalent non-trivial absolute values on $K$.  By abuse of 
notation, we call $\frak p$ a place of $\mathbb F_p(t)$.  We write $K_v$ for 
the completion of $K$ at the place $v$.  The set $M_K$ will always denote the 
collection of places of $K$ normalized in the following manner: for $\alpha \in 
K$ and $v$ lying over $\frak p$, define $$|\alpha|_v=|N_{K_v/F_{\frak 
p}}|_{\frak p}^{1/[K:F]}.$$  This normalization ensures that the product 
formula also holds for the set $M_K$. For a reference to the results stated 
so far in this section, see \cite{hdg} and \cite{stich}. 

For $\alpha \in K$, we define the \emph{(absolute logarithmic) height} of 
$\alpha$ as $$h(\alpha)=\frac{1}{[K:\mathbb{F}_p(t)]} \sum_{v\in M_K} 
\log\max(1, |\alpha|_v).$$  We extend this definition to $\mathbb P^1(K) = 
K\cup \{\infty\}$ by taking $h(\infty)=0$.  The following theorem is a well-
known and important feature of height functions.

\begin{thm}[Northcott's Theorem]
 Let $K$ be a finite extension of $\mathbb F_p(t)$ and $B$ be any non-negative 
real number.  Then there are finitely many $\alpha\in K$ such that 
$h(\alpha)\le B$.
\end{thm}

Nothcott's Theorem implies that the only points of height zero are roots of 
unity or zero. 

\begin{cor}
 Let $\alpha\in \overline{K}$.  Then $h(\alpha)=0$ if and only if $\alpha$ is a 
root of unity or $\alpha=0$.
\end{cor}
\begin{proof}
If $\alpha$ is a root of unity, then $h(\alpha)=0$ since $|\alpha|_v=1$ for 
each place $v$.  Suppose $h(\alpha)=0$.  Then since $h(\alpha^n)= nh(\alpha)$, 
we have that the set $\{\alpha, \alpha^2, \alpha^3, \dots \}$ is a set of 
bounded height.  Applying Northcott's Theorem over the finite extension 
$K(\alpha)$ implies that this set is finite.  Hence, $\alpha^n=\alpha^m$ for 
some $m$ and $n$.  If $\alpha\not=0$, then $\alpha^r=1$ for some $r$.
\end{proof}

\subsection{Integrality}
Let $S$ be a finite subset of $M_K$.  We say $\alpha \in \overline{K}$ is 
$S$-integral relative to $\beta\in\overline{K}$ if and only if for all 
$v\not\in S$ and embeddings $\sigma, \tau:\overline{K}/K\to \overline{K}_v$ we 
have 
\begin{align*}
&|\sigma(\beta) - \tau(\alpha)|_v \ge 1 &&\mbox{if} \; 
|\tau(\alpha)|_v \le 1\\ 
&|\sigma(\beta)|_v \le 1 &&\mbox{if} \; |\tau(\alpha)|_v > 1.
\end{align*}

This definition gives a generalization of the ring of integer in $K$, and 
it is symmetric:  $\alpha$ is $S$-integral relative to $\beta$ if and only if 
$\beta$ is $S$-integral relative to $\alpha$ (see \cite{sook}).  Enlarging the 
set of prime $S$ or extending the field $K$ enlarges the set of integral points. 
 We say $\alpha$ is $S$-integral relative to the set $D$ if and only 
if $\alpha$ is $S$-integral relative to each point in $D$.  

Let $v\in M_K$ and $\phi=[f(x,y):g(x,y)]$ be a rational map defined over $K$ 
with $a_1,\dots, a_n$ and $b_1,\dots, b_m$ the coefficients of $f$ and $g$, 
respectively.  We say $\phi$ is written in \emph{normalized form} if 
$\max(|a_1|_v,\dots,|a_n|_v,|b_1|_v,\dots,|b_m|_v)=1$.  Let $R_v=\{x \in K \mid 
|x|_v \le 1 \}$ be the valuation ring for $v$, $\mathfrak{m}_v=\{x\in K \mid 
|x|_v=1 \}$ be its maximal ideal, and $\kappa_v =
R_v/\mathfrak{m}_v$ be its residue field.  For $x\in R_v$, we say 
$\widetilde{x}$, the image of $x$ under the homomorphism $R_v \rightarrow 
\kappa_v$, is the reduction of $x$ modulo $\mathfrak{m}_v$.  Writing 
$\phi=[f:g]$ in normalized form, we let $\widetilde{\phi}$ be the 
rational map obtained by reducing the coefficients of $f$ and $g$ 
modulo $\mathfrak{m}_v$.  The map $\phi$ is said to have \emph{good reduction} 
at $v$ if $\deg(\phi)=\deg(\widetilde{\phi})$, and \emph{bad reduction} at $v$ 
otherwise.  This definition allows us to show that $S$-integrality 
behaves well functorially.

\begin{thm}\label{funct}
Suppose $\phi$ has good reduction for all places $v\not\in S$.  Then $\beta$ is 
$S$-integral relative to $\phi(\alpha)$ if and only if  $\phi^{-1}(\beta)$ is 
$S$-integral relative to $\alpha$.
\end{thm}

\begin{proof}
 See \cite{sook}, and note that Prop. 2.3 and Cor. 2.4 is valid 
over global fields.
\end{proof}

\section{Proof of Main Theorem}

The key to proving Theorem \ref{main} involves showing that if 
$|\alpha|_v=|\beta|_v=1$, then the points in $\phi^{-}(\beta)$ are not 
$v$-adically close to $\alpha$.  We do this by considering when $\beta=1$ 
and $\beta\not=1$. 

\begin{lem}\label{lem1}
Let $v$ be a place of $K$; fix an integer $d\ge2$ which is relatively prime to 
the characteristic of $K$, and let $\alpha \in K$ with $|\alpha|_v=1$. If $r$ is 
a real number satisfying $0<r<1$, then:
\begin{enumerate}[(i)]
 \item there are finitely many roots of unity $\zeta_{d^n} \in 
\overline{K}_v$ such that $0<|1-\zeta_{d^n}|_v<r$; and
 \item there is a constant $D$ such that $|\alpha-\zeta_{d^n}|_v>D$ for all 
roots of unity $\zeta_{d^n}\in \overline{K}_v$
\end{enumerate}
\end{lem}

\begin{proof}
 For (i), it suffices to show $|1-\zeta_{d^n}|_v\to 1$ as $n\to \infty$ 
since $|1-\zeta_{d^n}|_v\le 1$.  Assume $\zeta_{d^n} \in L$, a finite extension 
of $K$, and write $F=\mathbb{F}_p(t)$.  Let $f_{d^n}(x)$ be the monic 
irreducible polynomial of $\zeta_{d^n}$ over $F_{\frak p}$.  Then 
\begin{align*}
 |1-\zeta_{d^n}|_v &= |N_{L_v/F_{\frak 
p}}(1-\zeta_{d^n})|_{\frak p}^{1/[L:F]}\\
&=|f_{d^n}(1)|_{\frak p}^{1/[L:F]}
\end{align*}
Let $r$ be the radical of $d$, and write $\Phi_n(X)$ for the $n$th 
cyclotomic polynomial.  Note that $\Phi_{d^n}(X) = \Phi_r(X^{d^n/r})$.  Since 
$f_{d^n}(x)$ divides $\Phi_{d^n}(x)$, we have 
\begin{align*}
 |\Phi_{d^n}(1)|_{\frak p}&=|\Phi_r(1)|_{\frak p}\\
 &=|f_{d^n}(1)|_{\frak p}\cdot |h_{d^n}(1)|_{\frak p}.
\end{align*}
Since $r$ is relatively prime to the characteristic of $K$, we have 
$|\Phi_r(1)|_{\frak p}\not=0$, and this implies that $|f_{d^n}(1)|_{\frak p} 
\not\to 0$ as $n\to\infty$.  Therefore $|f_{d^n}(1)|^{1/[L:F]}_{\frak p} 
\to 1$ as $n\to \infty$ since $[L:F]\to \infty$.

To prove (ii), suppose for roots of unity $\zeta_{d^k}$ and $\zeta_{d^\ell}$, 
we have $|\alpha-\zeta_{d^k}|_v < r$ and $|\alpha-\zeta_{d^\ell}|_v < r$.  
The ultra-metric inequality implies $|\zeta_{d^k}-\zeta_{d^\ell}|_v < r$, or 
equivalently, $|1-\zeta_{d^n}|_v < r$ for some $n$.  By (i), there are 
finitely many $n$ such that $|1-\zeta_{d^n}|_v < r$ for any $0<r<1$.  Hence, 
$|\alpha-\zeta_{d^n}|_v<r$ for finitely many $n$; equivalently, there is a 
constant $D$ such that $|\alpha - \zeta_{d^n}|_v>D$ for all $\zeta_{d^n}$. 
\end{proof}

\begin{lem}\label{lem2}
Let $v$ be a place of $K$, and $\beta\in K$ with $|\beta|_v=1.$  If $d$ is 
relatively prime to the characteristic of $K$, then $|1-\beta^{d^n}|_v \not\to 
0$ and $n\to\infty$. 
\end{lem}

\begin{proof}
Assume $|1-\beta^{d^n}|_v \to 0$.  Then there is a sequence of positive 
integers $\{n_i\}_{i\ge 0}$ such that $|1-\beta^{d^{n_i}}|_v<1$ and 
$|1-\beta^{d^{n_i}}|_v\to 0$.  Taking $f_i(x)=1-x^{d^{n_i}}$, we have 
$|f_i(\beta)|_v<1$, and since $d$ is relatively prime to the characteristic, 
$|f_i'(\beta)|_v=1$.  By Hensel's Lemma (\cite{alain}), there exists roots of 
unity $\zeta_{d^{n_i}} \in \overline{K}_v$ such that $f_i(\zeta_{d^{n_i}})=0$, 
and $|\beta - \zeta_{d^{n_i}}|_v<|f_i(\beta)|_v \to 0$.  This contradicts 
Lemma \ref{lem1}(ii).  Hence $|1-\beta^{d^n}|_v \not\to 0$ as $n\to \infty$.
\end{proof}

\begin{lem}\label{lem3}
Let $v$ be a place of $K$ and $\phi(z)=z^d$ where $d\ge2$ is relatively prime to 
the characteristic of $K$.  Let $\alpha \in K$ with $|\alpha|_v=1$, and let 
$\beta \in K\setminus \{1\}$ with $|\beta|_v=1$.  If $r$ is a real number 
satisfying $0<r<1$, then:
\begin{enumerate}[(i)]
 \item there are at most finitely many $\xi \in \overline{K}_v$ such that 
$\xi^{d^m}=\beta^{d^n}$ for some positive integers $m > n$ satisfying 
$0<|1-\xi|_v<r$; and
 \item there is a constant $E$ such that $|\alpha-\gamma|_v>E$ for all $\gamma 
\in \phi^{-}(\beta)$.
\end{enumerate}
\end{lem}

\begin{proof}
Let $\xi_{m,n}$ for an $\xi \in L$ satisfying $\xi^{d^m}=\beta^{d^n}$ for 
some $m>n$, and write $w$ for the unique extension of $v$ to $L_w$.  
It suffices to show that $|1-\xi_{m,n}|_w\to 1$ as $n\to\infty$.  Let 
$g_{m,n}(x)$ be the irreducible polynomial of $\xi_{m,n}$ over $K_v$.  Then

\begin{align*}
|1-\xi_{m,n}|_w &= |N_{L_w/K_v}(1-\xi_{m,n})|_{v}^{1/[L:K]}\\
&=|g_{m,n}(1)|_v^{1/[L:K]}
\end{align*}

Since $g_{m,n}(x)$ divides $x^{d^m}-\beta^{d^n}$, we have
\begin{align*}
|1-\beta^{d^n}|_v = |g_{m,n}(1)|_v\cdot |k_{m,n}(1)|_v.
\end{align*}
By Lemma \ref{lem2}, $|1-\beta^{d^n}|_v \not\to 0$ and $n\to\infty$, and so 
$|g_{m,n}(1)|_v\not\to0$ as $n\to\infty$.  Hence  $|g_{m,n}(1)|_v^{1/[L:k]}\to 
1$ as $n\to \infty$ since $|g_{m,n}(1)|_v\le 1$ and $[L:K]\to\infty$ as 
$n\to\infty$.

For (ii), assume there are points $\gamma_1$ and $\gamma_2$ in 
$\phi^{-}(\beta)$ 
such that $|\alpha-\gamma_1|_v<r$ and $|\alpha-\gamma_2|_v<r$.  Then the 
ultrametric inequality implies $|\gamma_1-\gamma_2|_v<r$.  Since 
$|\gamma_1|_v=|\gamma_2|_v=1$, then we can write the previous inequality as 
$|1-\xi_{m,n}|<r$ where $\xi_{m,n}^{d^m}=\beta^{d^n}$ with $m>n$.  By (i), 
there are at most finitely $m,n$ such that $|1-\xi_{m,n}|<r$, and hence there 
at most finitely many $\gamma\in\phi^{-}(\beta)$ such that $|\alpha-\gamma|<r$. 
 Equivalently, there is a constant $E$ such that $|\alpha-\gamma|>E$ for 
all $\gamma\in \phi^{-}(\beta)$.
\end{proof}

We summarize the pertinent conclusions of Lemma \ref{lem1} and Lemma \ref{lem3} 
in the following proposition.

\begin{prop}\label{main-prop}
Let $v$ be a place of $K$; write $\phi(z)=z^d$ where $d\ge2$ is relatively 
prime to the characteristic of $K$; let $\alpha,\beta \in K$ with 
$|\alpha|_v=|\beta|_v=1$; and let $r$ be a real number satisfying $0<r<1$.  Then 
there are at most finitely $\gamma\in\phi^-(\beta)$ such that 
$0<|\alpha-\gamma|<r$. Equivalently, there is a constant $C$ such that 
$|\alpha-\gamma|_v>C$ for all $\gamma \in \phi^{-}(\beta)$.  
\end{prop}

We are now ready to prove the main theorem.

\begin{proof}[Proof of Theorem \ref{main}]
The goal is to show that if there are infinitely many $\gamma_n$, with 
$\gamma_n^{d^n}=\beta$, which are $S$-integral relative to $\alpha$, then
the sum $$A=\sum_{v\in 
M_K}\lim_{n\to\infty}\frac{1}{[K(\gamma_n):\mathbb F_p(t)]}\sum_{\sigma: 
K(\gamma_n)/K\to\overline{K}_v}\log|\sigma(\gamma_n)-\alpha|_v$$ equals 
$h(\alpha)$, the height of the non-preperiodic point $\alpha$.  Since 
$h(\alpha)>0$, this will contradict the fact that the product formula and 
integrality implies $A=0$.  Indeed, let $T=\{v\in M_K \mid |\alpha|_v >1 \}$.  
Then by integrality, $|\sigma(\gamma_n)-\alpha|_v=1$ for all $v\not\in S\cup 
T$, and this gives

\begin{align*}
A &=  \sum_{v\in 
S\cup 
T}\lim_{n\to\infty}\frac{1}{[K(\gamma_n):\mathbb F_p(t)]}\sum_{
\sigma}\log|\sigma(\gamma_n)-\alpha|_v\\
&= \lim_{n\to\infty} \frac{1}{[K(\gamma_n):\mathbb F_p(t)]} 
\sum_{\sigma}\sum_{v\in S\cup T} \log|\sigma(\gamma_n)-\alpha|_v\\
&=\lim_{n\to\infty} \frac{1}{[K(\gamma_n):\mathbb F_p(t)]} 
\sum_{\sigma}\,\log\,\left(\prod_{v\in M_K}|\sigma(\gamma_n)-\alpha|_v\right)\\
&=0
\end{align*}
To show $A=h(\alpha)$, we consider the cases when $|\beta|_v=1$ and 
$|\beta|_v\not=1$.

Suppose $|\beta|_v\not=1$.  Then $|\sigma(\gamma_n)|_v \not= |\alpha|_v$ for 
all $n$ sufficiently large.  The ultrametric inequality then implies 
$$\log|\sigma(\gamma_n)-\alpha|_v = \log\max (|\sigma(\gamma_n)|_v,|\alpha|_v) 
$$ for all $n$ sufficiently large.  Now $|\sigma(\gamma_n)|_v$ remains 
unchanged when $n$ is fixed and $\sigma$ varies since 
$\sigma(\gamma^{d^n})=\beta$.  Therefore 
$$\sum_{\sigma}\log|\sigma(\gamma_n)-\alpha|_v = [K(\gamma_n):K]\cdot \log\max 
(|\sigma(\gamma_n)|_v, |\alpha|_v),$$ and since 
$\lim_{n\to\infty}|\sigma(\gamma_n)|_v=1$, this gives
\begin{align*}
 \lim_{n\to\infty} \frac{1}{[K(\gamma_n):\mathbb F_p(t)]}\sum_{\sigma} 
\log|\sigma(\gamma_n)-\alpha|_v & = \frac{1}{[K:\mathbb F_p(t)]} \log\max (1, 
|\alpha|_v).
\end{align*}

Now suppose $|\beta|_v=1$.  If $|\alpha|_v\not=1$, we immediately get that 
$\log|\sigma(\gamma_n) - \alpha|_v = \log\max(1, |\alpha|_v)$.  Therefore, we 
only need to consider when $|\alpha|_v=1$.  By Proposition \ref{main-prop}, let 
$N(r)$ be the 
number of points $\sigma(\gamma_n)$ satisfying $0<|\sigma(\gamma_n)-\alpha|<r$ 
where $0<r<1$.  Then
\begin{align*}
0 &\ge \lim_{n\to\infty} \frac{1}{[K(\gamma_n):\mathbb F_p(t)]}\sum_{\sigma} 
\log|\sigma(\gamma_n)-\alpha|_v \\ &\ge \lim_{n\to\infty} 
\frac{1}{[K(\gamma_n):\mathbb F_p(t)]} \left\{ ([K(\gamma_n):K] - N(r))\log(r) 
+ 
N(r)\log(C)  \right\}\\
&=\frac{\log r}{[K:\mathbb F_p(t)]}
\end{align*}
Taking $r\to 1$ gives $$\lim_{n\to\infty} \frac{1}{[K(\gamma_n):\mathbb 
F_p(t)]}\sum_{\sigma} \log|\sigma(\gamma_n)-\alpha|_v = 0 = \log\max (1, 
|\alpha|_v).$$  This concludes the proof that $$A=\frac{1}{[K:\mathbb F_p(t)]} 
\sum_{v\in M_K} \log\max (1, |\alpha|_v) = h(\alpha).$$
\end{proof}

\section{Chebychev Polynomials}\label{cheby}
An immediate corollary of Theorem \ref{main}, obtained via Theorem \ref{funct}, 
is that the Conjecture \ref{bc} is true for \emph{Chebychev polynomials} 
arising for the maps $\phi(z)=z^d$ where $d$ is relatively prime to the 
characteristic of $K$.  Following \cite[Ch. 6]{ADS}, we define the Chebychev 
polynomials as maps $T_d$ making the following diagram commute
$$
\begin{CD}
\mathbb G_m @>z^d>> \mathbb G_m \\
@VV\pi V  @VV\pi V \\
\P^1 @>T_d>> \P^1
\end{CD}
$$
where $\pi$ is a finite $K$-morphism. 

\begin{cor}
Conjecture \ref{bc} is true for Chebyshev polynomials $T_d$ where $d$ is 
relatively prime to the characteristic of $K$.
\end{cor}

\begin{proof}
Using Theorem \ref{funct}, the proof is identical to the one given in \cite[Cor 
3.5]{sook}.
\end{proof}

\bibliographystyle{amsalpha}
\bibliography{back_orb}

\end{document}